\documentclass[12pt]{amsart}
\usepackage{amssymb}

\numberwithin{equation}{section}

\newtheorem{theorem}{Theorem}[section]
\newtheorem{defi}{Definition}[section]

\newtheorem{lemma}{Lemma}[section]

\newcommand{\N}{\mathbb{N}}
\newcommand{\Z}{\mathbb{Z}}
\newcommand{\Q}{\mathbb{Q}}

\newcommand{\G}{{G}}
\newcommand{\A}{\mathcal{A}}
\newcommand{\B}{\mathcal{B}}
\newcommand{\C}{\mathcal{C}}

\newcommand{\F}{\mathcal{F}}
\newcommand{\g}{\mathcal{G}}
\newcommand{\s}{\mathcal{S}}

\title[Decompositions of sets of integers]{On additive and multiplicative decompositions of sets of integers\\ with restricted prime factors, I.\\ (Smooth numbers.)}

\author[K. Gy\H{o}ry, L. Hajdu and A. S\'ark\"ozy]{K. Gy\H{o}ry, L. Hajdu and A. S\'ark\"ozy}

\thanks{Research supported in part by the NKFIH grants K115479, K119528, K128088, and K130909, and by the projects EFOP-3.6.1-16-2016-00022 and EFOP-3.6.2-16-2017-00015 of the European Union, co-financed by the European Social Fund.}



\address{K. Gy\H{o}ry
\hfill\break\indent L. Hajdu
\hfill\break\indent University of Debrecen, Institute of Mathematics
\hfill\break\indent H-4002 Debrecen, P.O. Box 400.
\hfill\break\indent Hungary}

\email{gyory@science.unideb.hu}
\email{hajdul@science.unideb.hu}

\address{A. S\'ark\"ozy
\hfill\break\indent E\"otv\"os Lor\'and University, Institute of Mathematics
\hfill\break\indent H-1117 Budapest, P\'azm\'any P\'eter s\'et\'any 1/C
\hfill\break\indent Hungary}

\email{sarkozy@cs.elte.hu}

\begin{document}

\begin{abstract}
In \cite{sa} the third author of this paper presented two conjectures on the additive decomposability of the sequence of ''smooth'' (or ''friable'') numbers. Elsholtz and Harper \cite{eh} proved (by using sieve methods) the second (less demanding) conjecture. The goal of this paper is to extend and sharpen their result in three directions by using a different approach (based on the theory of $S$-unit equations).
\end{abstract}

\maketitle

\section{Introduction}

$\A,\B,\C,\hdots$ denote (usually infinite) sets of non-negative integers, and their counting functions are denoted by $A(X),B(X),C(X),\hdots$ so that e.g.
$$
A(X)=|\{a:a\leq X,\ a\in\A\}|.
$$
The set of the positive integers is denoted by $\N$, and we write $\N\cup\{0\}=\N_0$. The set of rational numbers is denoted by $\Q$.

We will need
\begin{defi}
\label{defi1}
Let $\G$ be an {\rm additive} semigroup and $\A,\B,\C$ subsets of $\G$ with
\begin{equation}
\label{eq11}
|\B|\geq 2,\ \ \ |\C|\geq 2.
\end{equation}
If
\begin{equation}
\label{eq12}
\A=\B+\C\ (=\{b+c:b\in\B,\ c\in\C\})
\end{equation}
then \eqref{eq12} is called an {\rm additive decomposition} or briefly {\rm a-decomposition} of $\A$, while if a {\rm multiplication} is defined in $\G$ and \eqref{eq11} and
\begin{equation}
\label{eq13}
\A=\B\cdot\C\ (=\{bc:b\in\B,\ c\in\C\})
\end{equation}
hold then \eqref{eq13} is called a {\rm multiplicative decomposition} or briefly {\rm m-decomposition} of $\A$.
\end{defi}

In \cite{ost1} and \cite{ost2} H.-H. Ostmann introduced some definitions concerning additive properties of sequences of non-negative integers and studied some related problems. The most interesting definitions are:

\begin{defi}
\label{defi2}
A finite or infinite set $\A$ of non-negative integers is said to be {\rm a-reducible} if it has an {\rm additive decomposition}
$$
\A=\B+\C\ \ \ \text{with}\ \ \ |\B|\geq 2,\ |\C|\geq 2
$$
(where $\B\subset\N_0$, $\C\subset\N_0$). If there are no sets $\B,\C$ with these properties then $\A$ is said to be {\rm a-primitive} or {\rm a-irreducible}.
\end{defi}

(More precisely, Ostmann used the terminology "reducible", "primitive", "irreducible" without the prefix a-. However, since we will study both additive properties and their multiplicative analogs thus to distinguish between them we will use a prefix a- in the additive case and a prefix m- in the multiplicative case.)

\begin{defi}
\label{defi3}
Two sets $\A,\B$ of non-negative integers are said to be {\rm asymptotically equal} if there is a number $K$ such that $\A\cap[K,+\infty)=\B\cap [K,+\infty)$ and then we write $\A\sim\B$.
\end{defi}

\begin{defi}
\label{defi4}
An infinite set $\A$ of non-negative integers is said to be {\rm totally a-primitive} if every $\A'$ with $\A'\subset \N_0$, $\A'\sim\A$ is {\rm a-primitive}.
\end{defi}

The multiplicative analogs of Definitions \ref{defi2} and \ref{defi4} are:

\begin{defi}
\label{defi5}
If $\A$ is an infinite set of {\rm positive} integers then it is said to be {\rm m-reducible} if it has a {\rm multiplicative decomposition}
$$
\A=\B\cdot\C\ \ \ \text{with}\ \ \ |\B|\geq 2,\ |\C|\geq 2
$$
(where $\B\subset\N$, $\C\subset\N$). If there are no such sets $\B,\C$ then $\A$ is said to be {\rm m-primitive} or {\rm m-irreducible}.
\end{defi}

\begin{defi}
\label{defi6}
An infinite set $\A\subset\N$ is said to be {\rm totally m-primitive} if every $\A'\subset \N$ with $\A'\sim\A$ is {\rm m-primitive}.
\end{defi}

Many papers have been written on the existence or non-existence of a-decompositions and m-decompositions, resp., of certain special sequences; surveys of results of this type are presented in \cite{el, eh, hs1, hs3}. In \cite{sa} the third author of this paper presented two related conjectures (we adjust the original notation and terminology to match better to the ones used by Elsholtz and Harper who have proved related results in \cite{eh} later):

\begin{defi}
\label{defi7}
Denote the greatest prime factor of the positive integer $n$ by $p^+(n)$. Then $n$ is said to be smooth (or friable) if $p^+(n)$ is ''small'' in terms of $n$. More precisely, if $y=y(n)$ is a monotone increasing function on $\N$ assuming positive values and $n\in\N$ is such that $p^+(n)\leq y(n)$, then we say that $n$ is $y$-smooth, and we write $\F_y$ ($\F$ for ''friable'') for the set of all $y$-smooth positive integers.
\end{defi}

We quote \cite{sa} (using a slightly different notation):

\vskip.2cm

\noindent
''{\bf Conjecture A.} {\sl If $0<\varepsilon<1$,
$$
y(n)=n^\varepsilon,
$$
the set $\F_y\subset\N$ is defined by
$$
\F_y=\{n:n\in\N, p^+(n)\leq y(n)=n^\varepsilon\}
$$
and $\F_y'\subset\N$ is a set such that
$$
\F_y'\sim \F_y,
$$
then there are no sets $\A,\B\subset\N$ with $|\A|,|\B|\geq 2$ and
$$
\A+\B=\F_y'.
$$
}

\noindent
(...) this seems to be very difficult, but, perhaps, the ternary version of the problem can be settled:

\vskip.2cm

\noindent
{\bf Conjecture B.} {\sl If $\F_y$ and $\F_y'$ are defined as in Conjecture A, then there are no $\A,\B,\C\in\N$ with $|\A|,|\B|,|\C|\geq 2$ and
$$
\A+\B+\C=\F_y'.''
$$
}

Elsholtz and Harper (see Corollary 2.2 in \cite{eh}) proved Conjecture B for all small $\varepsilon>0$:

\vskip.2cm

\noindent
{\bf Theorem A.} {\sl There exists a large absolute constant $D>0$, and a small absolute constant $\kappa>0$, such that the following is true. Suppose $y(n)$ is an increasing function such that
\begin{equation}
\label{eq14}
(\log n)^D\leq y(n)\leq n^\kappa\ \ \ \text{\sl for large}\ n,
\end{equation}
and such that
$$
y(2n)\leq y(n)(1+(100\log y(n))/\log n).
$$
Then a ternary decomposition
$$
\A+\B+\C\sim\F_y',
$$
where $\A,\B$ and $\C$ contain at least two elements each, does not exist.}

\vskip.2cm

(This proves Conjecture B for $0<\varepsilon\leq\kappa$.)

In \cite{eh} first they proved ''an additive irreducibility theorem for sets that need not be well controlled by the sieve'' and then they deduced Theorem A from this theorem. In this paper our goal is to extend and sharpen their result in three directions: we will consider the decomposability of sets $\F_y'\sim\F_y$ with $y(n)$ smaller than the lower bound in \eqref{eq14}; in this case we will be able to also attack the more difficult problem of binary decomposition considered in Conjecture A; we will also study the multiplicative analog of the problem. While in \cite{eh} mostly sieve methods are used, here we will apply a completely different approach, namely, the crucial tool used by us will be the theory of $S$-unit equations.

Here we will prove the following two theorems:

\begin{theorem}
\label{thm1}
If $y(n)$ is an increasing function with $y(n)\to\infty$ and
\begin{equation}
\label{eq15}
y(n)<2^{-32}\log n\ \ \ \text{for large}\ n,
\end{equation}
then the set $\F_y$ is totally {\rm a-primitive}.
\end{theorem}

(If $y(n)$ is increasing then the set $\F_y$ is m-reducible since $\F_y=\F_y\cdot\F_y$, and we also have $\F_y\sim\F_y\cdot\{1,2\}$, thus if we want to prove an {\sl m-primitivity} theorem involving $\F_y$ then we have to switch from $\F_y$ to the shifted set
\begin{equation}
\label{eq16}
\g_y:=\F_y+\{1\}.
\end{equation}
See also \cite{el}.)

\begin{theorem}
\label{thm2}
If $y(n)$ is defined as in Theorem \ref{thm1}, then the set $\g_y$ is totally {\sl m-primitive}.
\end{theorem}

(While in part II of this paper we will present further closely related results.)

\section{Proof of Theorem \ref{thm1}}

Assume that contrary to the statement of the theorem, the function $y=y(n)$ satisfies the assumptions in Theorem \ref{thm1}, however, the set $\F_y$ is {\sl not} totally a-primitive. Then there are $\F_y'\subset \N_0$, $n_0\in\N$, $A\subset \N_0$, $\B\subset\N_0$ such that
\begin{equation}
\label{eq21}
\F_y'\cap [n_0,+\infty)=\F_y\cap [n_0,+\infty),
\end{equation}
\begin{equation}
\label{eq22}
|\A|\geq 2,\ |\B|\geq 2
\end{equation}
and
\begin{equation}
\label{eq23}
\F_y'=\A+\B.
\end{equation}
Let $N$ denote a positive integer with
\begin{equation}
\label{eq24}
N\to+\infty.
\end{equation}
It follows from \eqref{eq21} and \eqref{eq23} that
$$
\F_y\cap [n_0,N]=\F_y'\cap [n_0,N]=(\A+\B)\cap [n_0,N]\subset(\A\cap [0,N])+(\B\cap [0,N])
$$
whence
\begin{multline}
\label{eq25}
|\F_y\cap [n_0,N]|\leq |(\A\cap [0,N])+(\B\cap [0,N])|\leq\\
\leq |(\A\cap [0,N])|\cdot |(\B\cap [0,N])|=A(N)B(N).
\end{multline}
On the other hand, using the standard notation
$$
\Psi(x,y)=|\{n:n\leq x,\ p^+(n)\leq y\}|,
$$
for $N\to +\infty$ we have
\begin{multline}
\label{eq26}
|\F_y\cap [n_0,N]|=|\F_y\cap (0,N]|-|\F_y\cap (0,n_0)|\geq\\
\geq\Psi(N,y(N))-n_0=(1+o(1))\Psi(N,y(N))
\end{multline}
since clearly
\begin{equation}
\label{eq27}
\Psi(x,y)\to +\infty\ \text{for}\ 2\leq y\leq x,\ x\to +\infty.
\end{equation}
By \eqref{eq25} and \eqref{eq26}, for large enough $N$ we have
$$
\max(A(N),B(N))>{\frac{1}{2}} (\Psi(N,y(N)))^{1/2}.
$$
Thus either
$$
A(N)>{\frac{1}{2}} (\Psi(N,y(N)))^{1/2}
$$
or
\begin{equation}
\label{eq28}
B(N)>{\frac{1}{2}} (\Psi(N,y(N)))^{1/2}
\end{equation}
holds for infinitely many $N$; since $\A$ and $\B$ play symmetric roles thus we may assume that \eqref{eq28} does.

Write $\A=\{a_1,a_2,\dots\}$ with $(0\leq)a_1<a_2<\dots$ and
\begin{equation}
\label{eq29}
\tilde{\B}_N=\{b:b\in\B,\ n_0-a_1\leq b\leq N-a_2\}.
\end{equation}
Then by \eqref{eq27}, for all large enough $N$ satisfying \eqref{eq28} we have
\begin{multline}
\label{eq210}
|\tilde{\B}_N|=|(\B\cap [0,N])\setminus\left((\B\cap [0,n_0-a_1))\cup (\B\cap (N-a_2,N])\right)|\geq\\
\geq |\B\cap [0,N]|-|\B\cap [0,n_0-a_1)|-|\B\cap (N-a_2,N]|\geq\\
\geq B(N)-n_0-a_2>{\frac{1}{3}} (\Psi(N,y(N)))^{1/2}.
\end{multline}

We will need the notion of $S$-unit equations and a result on the number of solutions of them. For their formulation, we introduce some notation. Let $(0<)p_1<p_2<\dots<p_s$ be prime numbers, write $\s=\{p_1,p_2,\dots,p_s\}$ and let
$$
\Z_\s=\left\{\frac{a}{b}:a,b\in\Z,\ b\neq 0,\ (a,b)=1,\ p\mid b\implies p\in\s\right\}
$$
be the set of {\sl $\s$-integers}. Then the units of the ring $\Z_\s$, that is the set of {\sl $\s$-units} is given by 
\begin{equation}
\label{eq211}
\Z_\s^*=\left\{\frac{a}{b}:a,b\in\Z,\ ab\neq 0,\ (a,b)=1,\ p\mid ab\implies p\in\s\right\}.
\end{equation}

\begin{lemma}
\label{lem1}
If $U\in\Q$, $V\in\Q$ and $UV\neq 0$ then the {\rm $S$-unit equation}
\begin{equation}
\label{eq212}
UX+VY=1,\ X,Y\in\Z_\s^*
\end{equation}
has at most $2^{8(2s+2)}$ solutions.
\end{lemma}

\begin{proof} This assertion is a consequence of a theorem of Beukers and Schlickewei \cite{bs}; see Corollary 6.1.5 of Evertse and Gy\H{o}ry \cite{eg}, p.133.
\end{proof}

We will apply this lemma later with
\begin{equation}
\label{eq213}
\s=\{p:p\ \text{prime},\ p\leq y\}=\{p_1,p_2,\dots,p_{\pi(y)}\}
\end{equation}
(where $p_1<p_2<\dots<p_{\pi(y)}$ are the first $\pi(y)$ primes) so that now
\begin{equation}
\label{eq214}
s=|\s|=\pi(y)=\pi(y(N)).
\end{equation}
Consider now any
\begin{equation}
\label{eq215}
b\in \tilde{\B}_N,
\end{equation}
and write
\begin{equation}
\label{eq216}
X_b=a_2+b,\ Y_b=a_1+b.
\end{equation}
Then we have
$$
X_b-Y_b=a_2-a_1
$$
whence
\begin{equation}
\label{eq217}
\frac{1}{a_2-a_1}X_b-\frac{1}{a_2-a_1}Y_b=1.
\end{equation}
By \eqref{eq215} and \eqref{eq216} for all $b\in\tilde{\B}_N$ we have
\begin{multline}
\label{eq218}
n_0=a_1+(n-a_1)\leq a_1+b=Y_b<\\
<a_2+b=X_b\leq a_2+(N-a_2)=N,
\end{multline}
and by \eqref{eq23} we also have
\begin{equation}
\label{eq219}
a_1+b=Y_b\in\F_y'\ \text{and}\ a_2+b=X_b\in\F_y'.
\end{equation}
It follows from \eqref{eq21}, \eqref{eq218} and \eqref{eq219} that
$$
X_b,Y_b\in [n_0,N]\cap\F_y'=[n_0,N]\cap\F_y,
$$
thus $X_b,Y_b$ are composed from the primes not exceeding $y=y(N)$, i.e. from the set $\s$ defined in \eqref{eq213}, so that
\begin{equation}
\label{eq220}
X_b,Y_b\in\Z_\s^*
\end{equation}
(for the $\Z_\s^*$ defined in \eqref{eq211}). Writing $U=\frac{1}{a_2-a_1}$, $V=-\frac{1}{a_2-a_1}$, we have $U,V\in\Q$, thus
\begin{equation}
\label{eq221}
UX+VY=1,\ X,Y\in\Z_\s^*
\end{equation}
is an $S$-unit equation, and by \eqref{eq217} and \eqref{eq220} for every $b$ satisfying \eqref{eq215}, $X=X_b$, $Y=Y_b$ is a solution of this equation. It follows by \eqref{eq210} that the number $M$ of solutions of this equation satisfies
\begin{equation}
\label{eq222}
M\geq |\tilde{\B}_N|> \frac{1}{3} (\Psi(N,y(N)))^{1/2}.
\end{equation}

On the other hand, by Lemma \ref{lem1} and \eqref{eq214} we have
\begin{equation}
\label{eq223}
M\leq 2^{8(2s+2)}=2^{8(2\pi(y)+2)}.
\end{equation}
By \eqref{eq222} and \eqref{eq223} we have
\begin{equation}
\label{eq224}
\frac{1}{3} (\Psi(N,y(N)))^{1/2}<2^{8(2\pi(y)+2)}.
\end{equation}

Now we have to distinguish two cases.

CASE 1. Assume first that
\begin{equation}
\label{eq225}
2\leq y=y(N)\leq \log\log N.
\end{equation}
Then clearly we have
\begin{multline*}
\Psi(N,y(N))\geq \Psi(N,2)=|\{k\in\N_0,\ 2^k\leq N\}|=\\
=\left[\frac{\log N}{\log 2}\right]+1>\frac{\log N}{\log 2}>\log N
\end{multline*}
whence
\begin{equation}
\label{eq226}
\frac{1}{3} (\Psi(N,y(N)))^{1/2}>\frac{1}{3} (\log N)^{1/2}.
\end{equation}
On the other hand, by \eqref{eq225} we have
\begin{equation}
\label{eq227}
2^{8(2\pi(y)+2)}\leq 2^{8(2\pi(\log\log N)+2)}=2^{o(\log\log N)}=(\log N)^{o(1)}.
\end{equation}
\eqref{eq226} and \eqref{eq227} contradict \eqref{eq224}.

CASE 2. Assume now that
\begin{equation}
\label{eq228}
\log\log N<y(N)<2^{-32}\log N.
\end{equation}
We will need the following lemma:

\begin{lemma}
\label{lem2}
Write
$$
Z=\frac{\log x}{\log y}\log\left(1+\frac{y}{\log x}\right)+\frac{y}{\log y}\log\left(1+\frac{\log x}{y}\right).
$$
Then we have, uniformly for $x\geq y\geq 2$,
$$
\log\Psi(x,y)=Z\left(1+O\left(\frac{1}{\log y}+\frac{1}{\log\log 2x}\right)\right).
$$
\end{lemma}

\begin{proof}
This is de Bruijn's theorem \cite{br} (see also \cite{te} for the proof, background, and analysis of this formula).
\end{proof}

By \eqref{eq228} and Lemma \ref{lem2}, for $N$ large enough we have
\begin{multline}
\label{eq229}
\log\Psi(N,y(N))=Z\left(1+O\left(\frac{1}{\log y(N)}\right)\right)=\\
\left(\frac{\log N}{\log y(N)}\log\left(1+\frac{y(N)}{\log N}\right)+\frac{y(N)}{\log y(N)}\log\left(1+\frac{\log N}{y(N)}\right)\right)(1+o(1))>\\
>(1+o(1))\left(\frac{y(N)}{\log y(N)}\log(1+2^{32})\right).
\end{multline}

On the other hand, by \eqref{eq224}, \eqref{eq228} and the prime number theorem, for $N\to +\infty$ we have
\begin{multline*}
\log\Psi(N,y(N))<2\left(\log 3+\log 2^{8(2\pi(y)+2)}\right)=\\
=\log 9+2(2\pi(y)+2)\log 2^8=(1+o(1))\log 2^{32}\frac{y(N)}{\log y(N)}.
\end{multline*}
For $N$ large enough this contradicts \eqref{eq229} which completes the proof of Theorem \ref{thm1}.

\section{Proof of Theorem \ref{thm2}}
There are some similarities between the proofs of Theorems \ref{thm1} and \ref{thm2}, thus we will omit some details.

Assume that the conditions of Theorem \ref{thm2} hold, however, contrary to the statement of the theorem there are $\g_y'\subset\N$, $n_0\in\N$, $\A\subset\N$, $\B\subset\N$ such that
\begin{equation}
\label{eq31}
\g_y'\cap [n_0,+\infty)=\g_y\cap [n_0,+\infty),
\end{equation}
\eqref{eq22} holds, and
\begin{equation}
\label{eq32}
\g_y'=\A\cdot\B.
\end{equation}
Assume that $N\in\N$ satisfies \eqref{eq24}. Then it follows from \eqref{eq31} and \eqref{eq32} that
$$
\g_y\cap [n_0,N]=\g_y'\cap [n_0,N]\subset (\A\cap [0,N])\cdot (\B\cap [0,N])
$$
whence, by \eqref{eq16},
\begin{equation}
\label{eq33}
|\F_y\cap [n_0-1,N-1]|=|\g_y\cap [n_0,N]|\leq A(N)B(N).
\end{equation}
On the other hand, as in \eqref{eq26}, for $N\to +\infty$ we have
\begin{equation}
\label{eq34}
|\F_y\cap [n_0-1,N-1]|=(1+o(1))\F_y(0,N)=(1+o(1))\Psi(N,y(N)).
\end{equation}
By \eqref{eq33} and \eqref{eq34}, for every $N$ large enough we have
\begin{equation}
\label{eq35}
A(N)B(N)> \frac{1}{2} \Psi(N,y(N)).
\end{equation}

Now write $\A=\{a_1,a_2,\dots\}$ with $(0<)a_1<a_2<\dots$ and $\B=\{b_1,b_2,\dots\}$ with $(0<)b_1<b_2<\dots$, and define $m$ by $m=\max(a_2,b_2)$ (so that $m\geq 1$). We will show that there are infinitely many positive integers $D$ such that
\begin{equation}
\label{eq36}
A(mD)B(mD)< (m^2+1) A(D)B(D).
\end{equation}
Indeed, assume that contrary to this assertion there are only finitely many positive integers $D$ with this property. Then there exists a positive integer $D_0$ with 
\begin{equation}
\label{eq37}
A(D_0)B(D_0)>0
\end{equation}
such that for $D\in\N$, $D\geq D_0$ we have
$$
A(mD)B(mD)\geq (m^2+1) A(D)B(D).
$$
It follows from this by induction on $k$ that
\begin{equation}
\label{eq38}
A(m^kD_0)B(m^kD_0)\geq (m^2+1)^k A(D_0)B(D_0)\ \ \ \text{for}\ k=0,1,2,\dots .
\end{equation}
Clearly, we have
\begin{equation}
\label{eq39}
A(m^kD_0)B(m^kD_0)\leq m^kD_0\cdot m^k D_0=m^{2k}D_0^2 .
\end{equation}
We obtain from \eqref{eq38} and \eqref{eq39} that
$$
(m^2+1)^k A(D_0)B(D_0)\leq m^{2k}D_0^2
$$
whence
$$
\left(1+\frac{1}{m^2}\right)^k A(D_0)B(D_0)\leq D_0^2\ \ \ (\text{for}\ k=0,1,2,\dots).
$$
However, by \eqref{eq37}, this inequality cannot hold for $k$ large enough, and this contradiction proves the existence of infinitely many $D\in\N$ satisfying \eqref{eq36}.

Let $D$ be a positive integer satisfying \eqref{eq36} and large enough, and write $N=mD$. So far the sets $\A$ and $\B$ play symmetric roles thus we may assume that
\begin{equation}
\label{eq310}
B(D)\geq A(D).
\end{equation}
It follows from \eqref{eq35},\eqref{eq36} and \eqref{eq310} that
\begin{multline*}
\frac{1}{2} \Psi(N,y(N))<A(N)B(N)=A(mD)B(mD)<\\
<(m^2+1) A(D)B(D)\leq (m^2+1)(B(D))^2
\end{multline*}
whence, by $m\geq 1$,
\begin{equation}
\label{eq311}
B(D)\geq (2(m^2+1))^{-1/2} \left(\Psi(N,y(N))\right)^{1/2}\geq \frac{1}{2m}\left(\Psi(N,y(N))\right)^{1/2}.
\end{equation}
Now write
\begin{equation}
\label{eq312}
\tilde{\B}_N=\{b:b\in\B,\ n_0/a_1<b\leq N/a_2\}
\end{equation}
(note that $a_1\geq 1$ by $\A\subset\N$). Then by $\A\subset\N$, $N=mD$, the definition of $m$, and \eqref{eq311}, we have
\begin{multline}
\label{eq313}
|\tilde{\B}_N|=|\{b:b\in\B,\ n_0/a_1<b\leq N/a_2\}|=\\
=|\{b:b\in\B,\ b\leq N/a_2\}|-|\{b:b\in\B,\ b\leq n_0/a_1\}|\geq\\
\geq |\{b:b\in\B,\ b\leq N/m\}|-|\{b:b\in\B,\ b\leq n_0\}|=\\
=B(D)-B(n_0)\geq \frac{1}{2m} \left(\Psi(N,y(N))\right)^{1/2}-n_0 > \frac{1}{3m} \left(\Psi(N,y(N))\right)^{1/2}
\end{multline}
for $N$ large enough.

Consider now any
\begin{equation}
\label{eq314}
b\in\tilde{\B}_N
\end{equation}
and write
\begin{equation}
\label{eq315}
X_b=a_1b-1,\ \ \ Y_b=a_2b-1.
\end{equation}
Then we have
$$
a_2X_b-a_1Y_b=a_2(a_1b-1)-a_1(a_2b-1)=a_1-a_2,
$$
so that $X=X_b$, $Y=Y_b$ is a solution of the equation
\begin{equation}
\label{eq316}
\frac{a_2}{a_1-a_2} X-\frac{a_1}{a_1-a_2} Y=1.
\end{equation}
Moreover, by \eqref{eq312}, \eqref{eq314} and \eqref{eq315} we have
\begin{equation}
\label{eq317}
n_0\leq X_b=a_1b-1<Y_b=a_2b-1\leq N-1.
\end{equation}
It follows from \eqref{eq31} that
$$
a_1b\in\g_y'\ \ \ \text{and}\ \ \ a_2b\in\g_y',
$$
thus by \eqref{eq32} and \eqref{eq317} we also have
$$
a_1b\in\g_y\ \ \ \text{and}\ \ \ a_2b\in\g_y,
$$
so that by \eqref{eq16},
$$
X_b=a_1b-1\in\F_y\ \ \ \text{and}\ \ \ Y_b=a_2b-1\in\F_y.
$$
Thus $X_b$ and $Y_b$ satisfy \eqref{eq220} for the sets $\s$ and $\Z_\s^*$ defined by \eqref{eq213} and \eqref{eq211}, respectively. Then for every $b$ satisfying \eqref{eq314}, $X=X_b$, $Y=Y_b$ is a solution of the $S$-unit equation formed by \eqref{eq316} and $X,Y\in\Z_\s^*$ with this $\s,\Z_\s^*$, and clearly, if we start out from different $b$ values satisfying \eqref{eq314}, then we get different solutions $X=X_b$, $Y=Y_b$ of this equation. Thus by \eqref{eq313} the number $M$ of the solutions of this $S$-unit equation satisfies
\begin{equation}
\label{eq318}
M\geq |\tilde{\B}_N|> \frac{1}{3m} \left(\Psi(N,y(N))\right)^{1/2}.
\end{equation}
On the other hand, by Lemma \ref{lem1} the number of solutions must satisfy
\begin{equation}
\label{eq319}
M\leq 2^{8(2s+2)}=2^{8(2\pi(y)+2)}.
\end{equation}
It follows from \eqref{eq318} and \eqref{eq319} that
$$
\frac{1}{3m} \left(\Psi(N,y(N))\right)^{1/2}<2^{8(2\pi(y)+2)}.
$$
This is almost identical with inequality \eqref{eq224}, the only difference is that the constant factor $\frac{1}{3}$ on the left hand side of \eqref{eq224} is replaced here by $\frac{1}{3m}$ which is also independent of $N$, and thus it is easy to see that it leads to a contradiction in the same way as \eqref{eq224} did in Section 2.

\end{document}